\documentclass[11pt, a4paper]{article}
	
\usepackage{amsmath, amsthm, amscd, amsfonts, amssymb, graphicx, color}
\usepackage[bookmarksnumbered, plainpages]{hyperref}

 \newtheorem{thm}{Theorem}

\newtheorem{lem}{Lemma}

\begin{document}
\centerline{\Large{\bf Conditions for Bifurcations in a Non-Autonomous}}
\centerline{}
\centerline{\Large{\bf Scalar Differential Equations}}
 \centerline{}
\centerline{Sang-Mun Kim, Hyong-Chol O*}
 \centerline{}
 \small \centerline{Faculty of Mathematics, \textbf{Kim Il Sung} University, Pyongyang, D.P.R Korea}
\small \centerline{*e-mail address: hc.o@ryongnamsan.edu.kp}
\centerline{}
\centerline{}
\begin{abstract}
In this paper are provided some sufficient conditions for a non autonomous scalar differential equation to have saddle node, transcritical and pitchfork bifurcations using higher order derivatives.
\end{abstract}
{\bf Keywords:} non-autonomous scalar differential equation, saddle-node, transcritical, pitchfork, forwards attracting, pullback attracting \\
{\bf MSC(2010):} 37B55, 37C75, 37G10, 34C23
%
%
%
%
\section{Introduction}

\indent

The concept of non-autonomous dynamical systems can be said to have been made from the study on skew product flows
and random dynamical systems in 1990s in the viewpoint of topological dynamics.

A lot of developments have been made together with introduction of various types of concept of attractiveness and consideration of problems of existence and uniqueness of attracting sets and etc. \cite{cheb}-\cite{she}.

In \cite{lan2} they obtained sufficient conditions to occur transcritical, pitchfork and saddle-node bifurcations in a special
type of non-autonomous differential equation generalized from a canonical form of autonomous differential equation where transcritical,
pitchfork and saddle-node bifurcations occur.
And then using them, they studied the conditions for similar bifurcations in the general scalar non-autonomous equation
\begin{equation*}
\dot{x} = f(x, t, \lambda),
\end{equation*}
where $\lambda$ is a parameter. By imposing conditions on the Taylor coefficients in the expansion of $f$ near $x = \lambda = 0$,
they proved various general theorems guaranteeing transcritical, pitchfork, and saddle-node bifurcations.

In \cite{ras} they obtained a sufficient condition to occur transcritical bifurcation in non-autonomous differential equation
\begin{equation*}
\dot{x} = a(t, \alpha)x+b(t, \alpha)x^2+r(t,x, \alpha)
\end{equation*}
and a sufficient condition to occur pitchfork bifurcation in non-autonomous differential equation
\begin{equation*}
\dot{x} = a(t, \alpha)x+b(t, \alpha)x^3+r(t,x, \alpha).
\end{equation*}

In \cite{joh1, joh3} using the framework of skew product flows has been considered a generalized notion of a Hopf bifurcation
and in \cite{she} studied almost periodic scalar non - autonomous differential equations.
In \cite{klo2}, transcritical and pitchfork bifurcations in an almost periodic equation has been analyzed, \cite{joh2} has considered
a non-autonomous \textquoteleft two-step bifurcation\textquoteright and \cite{klo3} gave a nice discussion of the general problem
in the context of skew product flows.

On the other hand, conditions for bifurcations to occur in one-dimensional autonomous dynamical systems have been studied
using higher order derivatives. In \cite{val} sufficient conditions for transcritical, pitchfork, saddle-node and period doubling
bifurcations to occur in one-dimensional maps with one parameter have been studied using higher order derivatives.
In \cite{bal} sufficient conditions for cusp and period doubling bifurcations to occur in one-dimensional maps with two parameters have been studied using higher order derivatives.

In this paper we consider some non-autonomous differential equations generalized from autonomous dynamical systems in
\cite{bal, val}. First we try to obtain a sufficient condition to occur saddle-node and transcritical bifurcations in the equations
\begin{equation}
\dot{x} = \mu^{2m-1}f(t)-g(t)x^{2n}, ~ m, n \in \mathbf{N} \label{eq1}
\end{equation}
where the sufficient condition of \cite{lan2} does not satisfy.
And then we try to obtain sufficient conditions to occur saddle-node and transcritical bifurcations in more general equations
\begin{equation*}
\dot{x} = G(x, t, \lambda), ~ (\lambda \text{ is a parameter})
\end{equation*}
which include \eqref{eq1}.

%
%
%
%

\section{Preliminaries}

We consider the following initial value problem of non-autonomous differential equation
\begin{equation}
\dot{x} = f(t, x), x(s) = x_0 \label{eq2}
\end{equation}
defined on a domain $D \subset \mathbf{R}^m$ of $x$. In this case the initial time $s$ is as important as the main time variable $t$. Denote the solution to \eqref{eq2} by
\begin{equation*}
x(t, s; x_0) = S(t, s)x_0.
\end{equation*}
Then the {\it two-parameter family} $\{S(t, s)\}_{t \geq s}$ of {\it transformations in}  $D$ satisfying the following properties \cite{joh1, sch}:

1) $\forall t \in \textbf{R}, S(t, t)$ is the identity of $D$.

2) $S(t, \tau)S(\tau, s) = S(t, s) ~ (\forall t, \tau, s \in \textbf{R})$.

3) $S(t, s)x_0$ is continuous on $t, s, x_0$.
There may be solutions of \eqref{eq2} that do not exist for all time, and some restrictions to the possible values of $s$ and $t$ may be necessary, giving rise to only a ``local process''.
 
  Through the whole paper, we assume that $\{S(t,s):D\rightarrow D\}_{t\ge s}$ {\it preserves order} \cite{bal}, that is, $\{S(t,s):D\rightarrow D\}_{t\ge s}$ satisfies the property:
$$
x_s>y_s\Rightarrow S(t,s)x_s>S(t,s)y_s\qquad (\forall t,s\in {\mathbf R}).
$$

We first describe some needed basic concepts for the {\it two-parameter family} according to \cite{cheb} in what follows. 

A continuous map $x:{\mathbf R}\rightarrow {\mathbf R}^m$ is called a {\it complete trajectory} if it satisfies 
$$
S(t,s)x_s=x(t)~(\forall t,s\in {\mathbf R}).
$$

A time-varying family $\{\Sigma(t)\}_{t\in {\mathbf R}}$ of sets is called an {\it invariant set} for $S$ if it satisfies 
$$
S(t,s)\Sigma(s)=\Sigma(t)~(\forall t,s\in {\mathbf R}).
$$
If $\Sigma(t)\subset D$, $\forall t\in{\mathbf R}$, then we denote that $\Sigma(\cdot)\subset D$.

For two sets $A$ and $B$, the Hausdorff semi distance is defined as follows:
$$
dist[A,B]=\sup_{a\in A} \inf_{b\in B} d(a,b).
$$
From the definition, $dist[A,B]=0$ implies $A\subseteq B$.  

An invariant set $\Sigma(\cdot)$ is {\it forwards attracting} within $D$ if $\Sigma(\cdot)\subset D$ and
$$
\lim_{t\rightarrow\infty} dist[S(t,s)K,\Sigma(t)]=0
$$
for any $s\in {\mathbf R}$ and compact set $K\subset D$.

An invariant set $\Sigma(\cdot)$ is {\it locally forwards attracting} within $D$ if $\Sigma(\cdot)\subset D$ and for any $s\in {\mathbf R}$, there exists a $\delta(s)$ such that 
$$
\lim_{t\rightarrow\infty} dist[S(t,s)K,\Sigma(t)]=0
$$
for all compact $K\subset N(\Sigma(s),\delta(s))\cap D$.

If $\Sigma(\cdot)$ is forwards attracting in $D$, then it is locally forwards attracting in $D$.

An invariant set $\Sigma(\cdot)$ is {\it pullback attracting} within $D$ if $\Sigma(\cdot)\subset D$ and
$$
\lim_{s\rightarrow +\infty} dist[S(t,s)K,\Sigma(t)]=0
$$
for any $t\in {\mathbf R}$ and compact set $K\subset D$. 

$\Sigma(\cdot)$ is {\it globally} pullback attracting if we can take $D={\mathbf R}^m$.

An invariant set $\Sigma(\cdot)$ is {\it locally pullback attracting} within $D$ if $\Sigma(\cdot)\subset D$ and for any $t\in {\mathbf R}$, there exists a $\delta(t)$ such that if $K(\cdot)(\subset D)$ is compact and
$$
\lim_{s\rightarrow -\infty} dist[K(s),\Sigma(s)]<\delta(t)
$$
then
$$
\lim_{s\rightarrow -\infty} dist[S(t,s)K(s),\Sigma(t)]=0.
$$

If $D$ is a bounded set, any pullback attracting sets in $D$ is locally pullback attracting in $D$. If $D$ is a unbounded set, global pullback attracting sets in $D$ might not be locally pullback attracting in $D$. But we have the following useful results.
\begin{lem}[\cite{cheb}]
If an invariant set $\Sigma(\cdot)$ is pullback attracting set in $D$ and there is a $T$ such that $\bigcup_{t\le T} \Sigma(t)$ is bounded, then $\Sigma(\cdot)$ is locally pullback attracting.
\end{lem}

$\Sigma(\cdot)$ is pullback {\it Lyapunov stable} if 
$$
\forall t\in {\mathbf R},~\varepsilon>0,~\exists\delta(t,\varepsilon)>0;~\forall s(<t),~x_s\in U_{\delta(t)}(\Sigma(s))\Rightarrow S(t,s)x_s\in U_\varepsilon (\Sigma(t)).
$$
\begin{lem}[\cite{lan2}]
If $x^*(\cdot)$  is a complete trajectory and locally pullback attracting, then it is pullback Lyapunov stable. 
\end{lem}

$\Sigma(\cdot)$ is pullback unstable if it is not pullback Lyapunov stable, i.e. if 
$$
\exists t\in {\mathbf R},~\varepsilon>0,~\forall \delta>0;~\exists s<t,~x_s\in U_{\delta(t)}(\Sigma(s)),~dist[S(t,s)x_s,\Sigma(t)]>\varepsilon.
$$

If $\Sigma(\cdot)$ is an invariant set, then the {\it unstable set} $U_\Sigma (\cdot)$ of  is defined as
$$
U_\Sigma (s)=\{x_0:\lim_{t\rightarrow -\infty} dist[S(t,s)x_0,\Sigma(t)]=0\}.
$$

$\Sigma(\cdot)$ is {\it asymptotically unstable} if $\exists t;~U_\Sigma (t)\ne \Sigma(t)$.
\begin{lem}[\cite{lan3}]
If $\Sigma(\cdot)$ is asymptotically unstable, then it is pullback Lyapunov unstable and it cannot be locally pullback attracting.
\end{lem}

An invariant set $\Sigma(\cdot)$ is {\it pullback repelling} within $D$ if it is pullback attracting within $D$ for the time-reversed system, i.e. if $\Sigma(\cdot)\subset D$ and for any $t\in {\mathbf R}$ we have
$$
\lim_{s\rightarrow +\infty} dist[S(t,s)K,\Sigma(t)]=0
$$
for any compact set $K\subset D$. An invariant set $\Sigma(\cdot)$ is {\it locally pullback repelling} within $D$ if it is locally pullback attracting within $D$ for the time-reversed system, i.e. if $\Sigma(\cdot)\subset D$ and and for any $t\in{\mathbf R}$, there exists a  $\delta(t)$ such that if $K(\cdot)(\subset D)$ is compact and
$$
\lim_{s\rightarrow +\infty} dist[K(s),\Sigma(s)]<\delta(t)
$$
then
$$
\lim_{s\rightarrow +\infty} dist[S(t,s)K,\Sigma(t)]=0.
$$

An invariant set $\{A(t)\}_{t\in{\mathbf R}}$ is said to be the {\it pullback attractor} of the process $S$ within $D$ if it satisfies the following 3 conditions:

(1) For every $t\in{\mathbf R}$, $A(t)$ is a compact subset of $D$.

(2) $A(t)$ is pullback attracting within $D$.

(3) $A(t)$ is a minimal in the sense that if $\{C(t)\}_{t\in{\mathbf R}}$ is another family of closed sets that are pullback attracting within $D$, then $A(t)\subseteq C(t)$ for all $t\in{\mathbf R}$.

The following fact provides some information about the structure of attractor.
\begin{lem}[\cite{lan2}]
Let $\{K(t)\}_{t\in{\mathbf R}}$ be a family of non-empty compact sets and assume that for each $t_0$ and any compact set $B\subset D$ there exists a $T=T(t_0,B)$ such that $S(t_0,s)B\subset K(t_0)$, $\forall s\le T$. Then there exists a pullback attractor $A(t)$ within $D$, which is a connected set for every $t\in{\mathbf R}$. If $S(t,s)$ arises from a scalar ODE, then $A(t)=[a_-(t),a_+(t)]$ where $a_\pm (t)$ are complete trajectories for $S(t,s)$(\cite{lan2,lan3,lan5}). 
\end{lem}

The system $\dot{x}=G(x,t,\mu)$ undergoes a {\it local saddle node bifurcation} at $x=0$, $\mu=0$ if there exists $\mu_0>0,~\varepsilon>0$ and a $\delta$ with $0<\delta<\varepsilon$ such that 

(i) for $-\mu_0<\mu\le 0$, there are no complete trajectories lying within $(-\varepsilon,\varepsilon)$;

(ii) for $0<\mu<\mu_0$, there exists a complete trajectory $x_\mu ^+(t)$ that is pullback attracting within $(-\delta,\varepsilon)$ and another complete trajectory $x_\mu^-(t)$ that lies within $(-\varepsilon,\delta)$ and is asymptotically unstable. Furthermore, we have $x_\mu^\pm(t)\rightarrow 0(\mu\rightarrow 0)$.

The system $\dot{x}=G(x,t,\mu)$ undergoes a {\it local transcritical bifurcation} at $x=0$, $\mu=0$ if there exists $\mu_0>0,~\varepsilon>0$ such that

(i) for all $-\mu_0<\mu<0$, the zero solution is locally pullback attracting within $(-\varepsilon,0]$ and pullback attracting within $[0,\varepsilon)$; and there is another negative complete trajectory $x_\mu(t)$ within $(-\varepsilon,0)$ that is asymptotically unstable and $x_\mu(t)\rightarrow 0(\mu\rightarrow 0)$;

(ii) for $\mu=0$, the zero solution is asymptotically unstable but still pullback attracting within $[0,\varepsilon)$; and

(iii) for $0<\mu<\mu_0$, the zero solution is asymptotically unstable, and there is another positive complete trajectory $x_\mu(t)$ within $(0,\varepsilon)$ that is pullback attracting within $(0,\varepsilon)$ and $x_\mu(t)\rightarrow 0(\mu\rightarrow 0)$.

The system $\dot{x}=G(x,t,\mu)$ undergoes a {\it localised pitchfork bifurcation} at $x=0$, $\mu=0$ if there exists $\mu_0>0,~\varepsilon>0$ such that

(i) for all $-\mu_0<\mu\le 0$, the zero solution is pullback attracting within $(-\varepsilon,\varepsilon)$;

(ii) for $0<\mu<\mu_0$, the zero solution is asymptotically unstable, and there exist bounded complete trajectories $x_\mu^+(t)$ and $x_\mu^-(t)$ that are pullback attracting in $(0,\varepsilon)$ and $(-\varepsilon,0)$, respectively, and $x_\mu^\pm(t)\rightarrow 0(\mu\downarrow 0)$.

%
%

\section{Main Results}

\subsection{Saddle node Bifurcation}
First we consider a concrete example.
%
%
\begin{thm}\label{thr1}
Let consider the following non-autonomous differential equation
\begin{equation}
\dot{x} = \mu^{2m-1}f(t) - g(t)x^{2n}, ~ m, n \in \mathbf{N} \label{eq3}
\end{equation}
Let assume that $f(t)$ and $g(t)$ satisfy
\begin{eqnarray}
&& \int_{-\infty}^t f(s)ds = \int_t^{+\infty} f(s)ds = +\infty, \label{eq4} \\
&& \lim_{t \to \pm\infty}\textnormal{inf} ~ g(t)>0, ~ 0<l \leq \lim_{t \to \pm\infty} \frac{f(t)}{g(t)} \leq M. \label{eq5}
\end{eqnarray}
Then we have the following facts:

\textnormal{1)} When $\mu \leq 0$, non-zero bounded complete trajectories do not exist. When $\mu<0$, for any fixed $x_s$, we have
\begin{equation*}
\exists \sigma: s \leq \sigma, \exists t^*(s)<+\infty: \lim_{t \to t^*(s)} x(t, s; x_s) = -\infty
\end{equation*}
and for any fixed $t$, we have
\begin{equation*}
\exists s^*(t)>-\infty: \lim_{s \to s^*(t)} x(t, s; x_s) = -\infty.
\end{equation*}

\textnormal{2)} When $\mu=0$, the zero solution is locally pullback and forwards attracting within $[0, \infty)$. For the solution with initial value in $(-\infty, 0)$,
we have the same conclusions with the case of $\mu<0$.

\textnormal{3)} When $\mu>0$, there exist two trajectories $x^*(t)$ and $y^*(t)$ such that $x^*(t)$ is pullback and forwards attracting, that is, we have
\begin{eqnarray*}
&& \lim_{s \to -\infty}S(t, s)x_0 = x^*(t), ~ x_0>\sqrt[-2n]{l\mu^{2m-1}}, \\
&& \lim_{t \to +\infty} \textnormal{dist} \left[ S(t, s)x_0, x^*(t) \right] = 0, ~ x_0>\sqrt[-2n]{l\mu^{2m-1}},
\end{eqnarray*}
And $y^*(t)$ is pullback repelling and asymptotically instable, that is, we have
\begin{eqnarray*}
&& \lim_{s \to +\infty}S(t, s)x_0 = y^*(t), ~ x_0<\sqrt[2n]{l\mu^{2m-1}}, \\
&& \lim_{t \to -\infty} \textnormal{dist} \left[ S(t, s)x_0, y^*(t) \right] = 0, ~ x_0<\sqrt[2n]{l\mu^{2m-1}}.
\end{eqnarray*}
\end{thm}
\begin{proof}
In the case of $\mu<0$, by \eqref{eq5} we have
\begin{equation*}
\exists T: \forall t \leq T \Rightarrow f(t)>0, ~ g(t)>0.
\end{equation*}
In the case of $x_s<0$, by the above expression we have $\forall t \leq T, ~ \dot{x} \leq -g(t)x^{2n}$ and
\begin{eqnarray*}
&& \int_s^t \frac{\dot{x}}{x^{2n}}dr \leq -\int_s^t g(r)dr \Rightarrow \int_{x(s)}^{x(t)} \frac{1}{x^{2n}}dx \leq -\int_s^t g(r)dr \\
&& \left. \Rightarrow -\frac{1}{(2n-1)}x^{-(2n-1)} \right\vert_{x=x(s)}^{x(t)} \leq -\int_s^t g(r)dr \\
&& \Rightarrow -\frac{1}{(2n-1)}x(t)^{-(2n-1)} \leq -\frac{1}{(2n-1)}x(s)^{-(2n-1)}-\int_s^t g(r)dr \\
&& \Rightarrow x(t)^{(2n-1)} \leq \left(x(s)^{-(2n-1)}+(2n-1)\int_s^t g(r)dr\right)^{-1} \\
&& \Rightarrow x(t) \leq \left(x(s)^{-(2n-1)}+(2n-1)\int_s^t g(r)dr\right)^{-\frac{1}{2n-1}}.
\end{eqnarray*}
For fixed $t, x_s^{-(2n-1)} < 0$ and $g$ satisfies
\begin{equation*}
\exists T^*: s \leq T^* \Rightarrow \int_s^t g(r)dr>0.
\end{equation*}
Thus we have
\begin{equation*}
\exists s^*(t)>-\infty: \lim_{s \to s^*(t)} x(t, s; x_s) =-\infty.
\end{equation*}
Similarly, if $x_s$ is fixed, we have
\begin{equation*}
\exists \sigma(t): s \leq \sigma(t), \exists t^*(s)<+\infty: \lim_{t \to t^*(s)} x(t, s; x_s) =-\infty.
\end{equation*}
Let consider the case of $x_s<0$. Then for $x_s = -1$, we have
\begin{equation}
\exists \sigma_1: s \leq \sigma_1 \Rightarrow \exists t^*(s)<+\infty, \lim_{t \to t^*(s)} x(t, s; -1) =-\infty. \label{eq6}
\end{equation}
If $t\leq T$, then $\dot{x} \leq \mu^{(2n-1)}f(t)<0$ and thus we have
\begin{equation}
x(t, s; x_s) \leq x_s+\mu^{(2n-1)} \int_s^t f(r)dr, ~ t \leq T. \label{eq7}
\end{equation}
Using $\mu<0$, \eqref{eq4} and \eqref{eq7}, we have
\begin{equation*}
\exists \sigma_2: s \leq \sigma_2 \Rightarrow t \leq \sigma_1, x(t, s; x_s) \leq -1.
\end{equation*}
From \eqref{eq6} and property of order preservation we know $t \leq \sigma_1$ and thus
\begin{equation*}
\forall \tau>t, ~ x(\tau, t; x(t, s; x_s)) \leq x(\tau, t; -1).
\end{equation*}
When $\tau \to t^*(s)$, we have $x(\tau, t; -1) \to -\infty$ and thus $x(\tau, s; x_s) \to -\infty$. On the other hand, for fixed $t$, when $s \to s_1(t)>-\infty$, we have $x(t, s; -1) \to -\infty$.

Using \eqref{eq7}, we have $\exists s_2: s \leq s_2, x(t, s; x_s) \leq -1$ and from property of order perservation, we have $x(t, s; x_s) \leq x(t, s; -1)$. When $s \to s_1(s_2)$, we have $x(t, s; x_s) \to -\infty$.

Next consider the case of $\mu=0$. Then the solution of \eqref{eq3} is as follows:
\begin{equation*}
x(t, s; x_s) = \frac{1}{\left[x_s^{-(2n-1)} + (2n-1) \int_s^t g(r)dr \right]^{\frac{1}{2n-1}}}.
\end{equation*}
If $x_s \geq 0$, then $x_s^{-(2n-1)} \geq 0$ and from \eqref{eq4} we have $s \to -\infty (t \to +\infty)$. Then
\begin{equation*}
(2n-1) \int_s^t g(r)dr \to +\infty, ~ \left[x_s^{-(2n-1)} + (2n-1) \int_s^t g(r)dr \right]^{\frac{1}{2n-1}} \to +\infty.
\end{equation*}
Thus we have
\begin{equation*}
x(t, s; x_s) \to 0 (t \to +\infty, s \to -\infty).
\end{equation*}
In order to show that the zero solution is locally pullback attracting, we must prove
\begin{equation*}
\left[x_s^{-(2n-1)} + (2n-1) \int_s^t g(r)dr \right]^{\frac{1}{2n-1}} >0, ~ \forall \tau \in [s, t].
\end{equation*}
If
\begin{equation*}
x_s < \frac{1}{\sup_{\tau \in [T^-, t]} \left\vert \left[ (2n-1)\int_{T^-}^{\tau}g(r)dr \right]^{\frac{1}{2n-1}} \right\vert} = \delta(t),
\end{equation*}
then the above expression holds. Thus the zero solution is locally pullback attracting.
It is similar to prove that the zero solution is locally forwards attracting.

If $x_s<0$, then we have the same result with the case when $\mu<0$ and $x_s<0$.

Let consider the case of $\mu>0$. From the condition \eqref{eq5} we have
\begin{eqnarray*}
&& \exists T^-<0, \exists T^+>0; \forall t \leq T^-, \forall t \geq T^+ \\
&& \Rightarrow \dot{x} \leq \mu^{(2m-1)}Mg(t)-g(t)x^{2n} = g(t)\left[ M\mu^{(2m-1)}-x^{2n}\right], \\
&& \dot{x} \geq \mu^{(2m-1)}lg(t)-g(t)x^{2n} = g(t)\left[ l\mu^{(2m-1)}-x^{2n}\right].
\end{eqnarray*}
Thus we have
\begin{eqnarray*}
&& \dot{x} \leq g(t) \left[ \sum_{k=1}^n \left( \sqrt[2n]{M\mu^{2m-1}}\right)^{2(n-k)} x^{2(k-1)}\right] \left[\left( \sqrt[2n]{M\mu^{2m-1}}\right)^2-x^2 \right], \\
&& \dot{x} \geq g(t) \left[ \sum_{k=1}^n \left( \sqrt[2n]{l\mu^{2m-1}}\right)^{2(n-k)} x^{2(k-1)}\right] \left[\left( \sqrt[2n]{l\mu^{2m-1}}\right)^2-x^2 \right].
\end{eqnarray*}
Let
\begin{eqnarray*}
&& g_1(t):=g(t) \left[ \sum_{k=1}^n \left( \sqrt[2n]{M\mu^{2m-1}}\right)^{2(n-k)} x^{2(k-1)}\right], \\
&& g_2(t):=g(t) \left[ \sum_{k=1}^n \left( \sqrt[2n]{l\mu^{2m-1}}\right)^{2(n-k)} x^{2(k-1)}\right].
\end{eqnarray*}
Then $g_1(t)$ and $g_2(t)$ satisfy the condition \eqref{eq5} on $g(t)$. Therefore we have
\begin{eqnarray*}
&& \dot{x} \leq g_1(t) \left[ \sqrt[2n]{M\mu^{2m-1}} + x \right] \left[ \sqrt[2n]{M\mu^{2m-1}} - x \right], \\
&& \dot{x} \geq g_2(t) \left[ \sqrt[2n]{l\mu^{2m-1}} + x \right] \left[ \sqrt[2n]{l\mu^{2m-1}} - x \right].
\end{eqnarray*}
If $x_0 >  \sqrt[-2n]{l\mu^{2m-1}}$, then
\begin{equation*}
\sqrt[2n]{l\mu^{2m-1}} \leq \lim_{\substack{s \to -\infty \\ t \to +\infty}}x(t, s; x_0) \leq \sqrt[2n]{M\mu^{2m-1}}.
\end{equation*}

Now let $x_1(t)$ and $x_2(t)$ be two different solutions of \eqref{eq3} and $z(t)=x_1(t)-x_2(t)$. Then
\begin{equation*}
\dot{x}_1(t) = \mu^{2m-1}f(t) - g(t)x_1^{2n}(t), ~ \dot{x}_2(t) = \mu^{2m-1}f(t) - g(t)x_2^{2n}(t).
\end{equation*}
Thus we have
\begin{equation}
\dot{z}(t) = -g(t)\left[ x_1^{2n}-x_2^{2n}\right] = -g(t)\left[ \sum_{k=1}^n x_1^{2(n-k)}(t) x_2^{2(k-1)}(t)\right] [x_1+x_2]z(t). \label{eq8}
\end{equation}

If $t\le T^-$ or $\forall t\ge T^+$, then we have $g(t)\left(\sum_{k=1}^n x_1^{2(n-k)}(t)x_2^{2(k-1)}(t)\right)>0$ and $x_1(t),~x_2(t)\ge\sqrt[2n]{l\mu^{2m-1}}$, thus if $\forall t\le T^-$ or $\forall t\ge T$, then $x_1(t)=x_2(t)$. Therefore there exists a positive solution $x^*(t)$ such that it (pullback, forwards) attracts all trajectories with initial data greater than $\sqrt[-2n]{l\mu^{2m-1}}$.
Now if $x_0<\sqrt[-2n]{M\mu^{2m-1}}$, then the solutions go to $-\infty$ (pullback, forwards). If $x_0<\sqrt[2n]{l\mu^{2m-1}}$, then
$
\sqrt[-2n]{M\mu^{2m-1}} \leq \lim_{\substack{t \to -\infty \\ s \to +\infty}}x(t, s; x_0) \leq \sqrt[-2n]{l\mu^{2m-1}}
$
and for the two different solutions $x_1(t)$ and $x_2(t)$ of \eqref{eq3}, we have \eqref{eq8}. Repeating the above arguments, we have the following conclusion:

If $\forall t \leq T^-, ~ \forall t \geq T^+, ~ x_1(t), x_2(t) \leq \sqrt[-2n]{l\mu^{2m-1}}$, then $ x_1(t)=x_2(t)$.

Thus there exists a negative solution $y^*(t)$ such that it (pullback, forwards) attracts all trajectories with initial data less than $\sqrt[2n]{l\mu^{2m-1}}$ in the meaning of time inverse. That is, $y^*(t)$ is pullback repelling.
\begin{equation*}
\lim_{\substack{s \to +\infty \\ t \to -\infty}}x(t, s; x_s) = y^*(t), ~ x_s < \sqrt[2n]{l\mu^{2m-1}}.
\end{equation*}
\end{proof}

Now we consider general equations
\begin{equation}
\dot{x} = G(t, x, \mu). \label{eq9}
\end{equation}
Assume that $G$ is sufficiently smooth. The following is Taylor expansion of $G$ at $(t, 0, 0)$.
\begin{eqnarray*}
&& G(t, x, \mu) =  G(t, 0, 0)+G_x(t, 0, 0)x+G_{\mu}(t, 0, 0)\mu + \frac{1}{2}G_{xx}(t, 0, 0)x^2 \\
&& \qquad + G_{x\mu}(t, 0, 0)x\mu + \frac{1}{2}G_{\mu\mu}(t, 0, 0)\mu^2+ \frac{1}{6}G_{xxx}(t, 0, 0)x^3 + \frac{1}{2}G_{xx\mu}(t, 0, 0)x^2\mu \\
&& \qquad + \frac{1}{2}G_{x\mu\mu}(t, 0, 0)x\mu^2+ \frac{1}{6}G_{\mu\mu\mu}(t, 0, 0)\mu^3 + \cdots + \frac{1}{(2n)!} \left[ \frac{\partial^{2n}}{\partial x^{2n}}G(t, 0, 0)x^{2n} \right. \\
&& \qquad + C_{2n}^1 \frac{\partial^{2n}}{\partial x^{2n-1} \partial \mu}G(t, 0, 0)x^{2n-1}\mu + \cdots + C_{2n}^{2n-1} \frac{\partial^{2n}}{\partial x \partial \mu^{2n-1}}G(t, 0, 0)x \mu^{2n-1} \\
&& \qquad \left. +\frac{\partial^{2n}}{\partial \mu^{2n}}G(t, 0, 0)\mu^{2n} \right] + \frac{1}{(2n+1)!} \left[ \frac{\partial^{2n+1}}{\partial x^{2n+1}}G(t, 0, 0)x^{2n+1} \right.\\
&& \qquad + C_{2n+1}^1 \frac{\partial^{2n+1}}{\partial x^{2n} \partial \mu}G(t, 0, 0)x^{2n}\mu + \cdots + C_{2n+1}^{2n} \frac{\partial^{2n+1}}{\partial x \partial \mu^{2n}}G(t, 0, 0)x \mu^{2n}+ \cdots \\
&& \qquad \left. + C_{2n+1}^{2n} \frac{\partial^{2n+1}}{\partial x \partial\mu^{2n}}G(t, 0, 0)x\mu^{2n} + \frac{\partial^{2n+1}}{\partial \mu^{2n+1}}G(t, 0, 0)\mu^{2n+1} \right].
\end{eqnarray*}
Here $n \in \mathbf{N}$.

Now assume that $G$ satisfies the following conditions:
\begin{eqnarray*}
&& \textnormal{(i)} ~ G(t, 0, 0) = 0, ~ \forall t \in \mathbf{R}, \\
&& \textnormal{(ii)} ~ \frac{\partial}{\partial x}G(t, 0, 0) = \frac{\partial^2}{\partial x^2}G(t, 0, 0) = \cdots = \frac{\partial^{2n-1}}{\partial x^{2n-1}}G(t, 0, 0) = 0. \qquad \qquad \qquad \qquad
\end{eqnarray*}
Then $G$ is provided as follows:
\begin{eqnarray*}
&& G(t, x, \mu) =  \mu \left[ G_{\mu}(t, 0, 0)+G_{x\mu}(t, 0, 0)x+ \frac{1}{2}G_{\mu\mu}(t, 0, 0)\mu + \frac{1}{3}G_{xx\mu}(t, 0, 0)x^2 \right. \\
&& \quad + \frac{1}{6}G_{\mu\mu}(t, 0, 0)\mu^2 + \frac{1}{3}G_{x\mu\mu}(t, 0, 0)x\mu + \cdots + \frac{1}{(2n)!}C_{2n}^1 \frac{\partial^{2n}}{\partial x^{2n-1} \partial \mu} \\
&& \quad  \times ~ G(t, 0, 0)x^{2n-1} + \cdots + \frac{1}{(2n)!}C_{2n}^{2n-1} \frac{\partial^{2n}}{\partial x \partial \mu^{2n-1}}G(t, 0, 0)x \mu^{2n-2} \\
&& \quad \left. + \frac{1}{(2n)!}\frac{\partial^{2n}}{\partial \mu^{2n}}G(t, 0, 0)\mu^{2n-1} + \frac{1}{(2n+1)!} C_{2n+1}^1 \frac{\partial^{2n+1}}{\partial x^{2n} \partial \mu}G(t, 0, 0)x^{2n} + \cdots \right] \\
&& \quad +\left[ \frac{1}{(2n)!}\frac{\partial^{2n}}{\partial x^{2n}}G(t, 0, 0) + \frac{1}{(2n+1)!} \frac{\partial^{2n+1}}{\partial x^{2n+1}}G(t, 0, 0)x + \cdots \right]x^{2n}+ \cdots.
\end{eqnarray*}

If we denote $f(t):=G_{\mu}(t, 0, 0)$ and $g(t) := -\frac{1}{(2n)!} \frac{\partial^{2n}}{\partial x^{2n}}G(t, 0, 0)$, then the system \eqref{eq9} can be rewritten as follows:
\begin{equation}
\dot{x} = \mu[f(t)+\phi(t, x, \mu)] - x^{2n}[g(t)+\psi(t, x)]. \label{eq10}
\end{equation}
Here $\phi(t, 0, 0) = 0, ~ \psi(t, 0) = 0$.
%
%
\begin{thm} \label{thr2}
Assume that $G$ satisfies the above mentioned conditions and
\begin{equation}
\lim_{t \to \pm\infty}\textnormal{inf} ~ g(t)>0, \label{eq11}
\end{equation}
\begin{equation}
0<m = \lim_{t \to \pm\infty}\textnormal{inf} ~ \frac{f(t)}{g(t)} \leq  \lim_{t \to \pm\infty}\textnormal{sup} ~ \frac{f(t)}{g(t)}=M<+\infty, \label{eq12}
\end{equation}
and there exists a positive valued function $h(t)$ such that
\begin{equation}
\vert \phi(t, x, \mu) \leq h(t)[\vert x\vert+ \vert\mu\vert], ~~ \vert\phi_x(t, x, \mu) \vert \leq h(t), \label{eq13}
\end{equation}
\begin{equation}
\vert \psi_x(t, x) \vert \leq h(t), \label{eq14}
\end{equation}
\begin{equation}
\lim_{t \to \pm\infty}\textnormal{sup} ~ \frac{h(t)}{g(t)} \leq k. \label{eq15}
\end{equation}
Then there occurs local saddle-node bifurcation when $\mu$ passes through 0.
Furthermore, when $\mu>0$, locally attracting trajectory $x_{\mu}(t)$ is forwards attracting in $(0, \varepsilon)$ and unstable
trajectories are pullback repelling in $(-\varepsilon, \delta)$.
\end{thm}

The  main idea of the proof  is  similar to Theorem \ref{thr1} and omitted. \\

Now assume that $G$ satisfies the following conditions:
\begin{eqnarray*}
&& \textnormal{(i)} ~ G(t, x, \mu) = G(t, x, 0)+c(x)G(0, 0, \mu), \\
&& \textnormal{(ii)} ~ G(t, 0, 0) = 0, \\
&& \textnormal{(iii)} ~ \frac{\partial}{\partial x}G(t, 0, 0) = \frac{\partial^2}{\partial x^2}G(t, 0, 0) = \cdots = \frac{\partial^{2n-1}}{\partial x^{2n-1}}G(t, 0, 0)=0, \\
&& \textnormal{(iv)} ~ \frac{\partial}{\partial \mu}G(t, 0, 0) = \frac{\partial^2}{\partial \mu^2}G(t, 0, 0) = \cdots = \frac{\partial^{2m-2}}{\partial \mu^{2m-2}}G(t, 0, 0)=0. \qquad \qquad \qquad \qquad
\end{eqnarray*}
Here $n, m \in \mathbf{N}$. Then Tailor expansion of $G$ is provided as follows:
\begin{eqnarray*}
&& G(t, x, \mu) =  \mu^{2m-1} \left[x\mu^{-(2m-2)}  \frac{\partial^2}{\partial x \partial\mu} G(t, 0, 0)+\frac{1}{2}x^2\mu^{-(2m-2)}\frac{\partial^3}{\partial x^2 \partial\mu}G(t, 0, 0) \right. \\
&& \qquad + \frac{1}{2}x\mu^{-(2m-3)}\frac{\partial^3}{\partial x \partial\mu^2}G(t, 0, 0) + \cdots +\frac{1}{(2m-1)!}\frac{\partial^{2m-1}}{\partial\mu^{2m-1}}G(t, 0, 0) \\
&& \qquad + \frac{1}{(2m)!}C_{2m}^1 x^{2m-1}\mu^{-(2m-2)} \frac{\partial^{2m}}{\partial x^{2m-1} \partial\mu}G(t, 0, 0) \\
&& \qquad \left. + \frac{1}{(2m)!}C_{2m}^2 x^{2m-2}\mu^{-(2m-3)} \frac{\partial^{2m}}{\partial x^{2m-2} \partial\mu^2}G(t, 0, 0)+\cdots \right] \\
&& \qquad + x^{2n}\left[ \frac{1}{(2n)!} \frac{\partial^{2n}}{\partial x^{2n}}G(t, 0, 0)+\frac{1}{(2n+1)!}x\cdot \frac{\partial^{2n+1}}{\partial x^{2n+1}}G(t, 0, 0)+\cdots \right].
\end{eqnarray*}
Denote $f(t):=\frac{1}{(2m-1)!}\frac{\partial^{2m-1}}{\partial \mu^{2m-1}}G(t, 0, 0), ~ g(t):=-\frac{1}{(2n)!}\frac{\partial^{2n}}{\partial x^{2n}}G(t, 0, 0)$. Then we can rewritten \eqref{eq9} as follows:
\begin{equation*}
\dot{x} = \mu^{2m-1}[f(t)+\phi(t, x, \mu)] - x^{2n}[g(t)+\psi(t, x)].
\end{equation*}
Here
\begin{equation*}
\phi(t, 0, 0) = 0, ~ \psi(t, 0) = 0.
\end{equation*}
%
%
\begin{thm}
Assume that
\begin{equation*}
\lim_{t \to \pm\infty}\textnormal{inf} ~ g(t)>0, ~ 0<m = \lim_{t \to \pm\infty}\textnormal{inf}\frac{f(t)}{g(t)} \leq  \lim_{t \to \pm\infty}\textnormal{sup} ~ \frac{f(t)}{g(t)}=M<+\infty,
\end{equation*}
\begin{equation*}
\vert \phi(t, x, \mu) \leq h(t)[\vert x\vert+ \vert\mu\vert^{-(2m-2)}], ~~ \vert\phi_x(t, x, \mu) \vert \leq h(t),
\end{equation*}
\begin{equation*}
\vert \psi_x(t, x) \vert \leq h(t), ~~ \lim_{t \to \pm\infty}\textnormal{sup} ~ \frac{h(t)}{g(t)} \leq k.
\end{equation*}
Then there occurs local saddle-node bifurcation when $\mu$ passes through 0. Furthermore, when $\mu>0$,
locally attracting trajectory $x_{\mu}(t)$ is forwards attracting in $(0, \varepsilon)$ and unstable trajectory are pullback repelling in $(-\varepsilon, \delta)$.
\end{thm}
The proof is omitted.\\

\textbf{Example 1}. In the equation $\dot{x} = \mu^3t^2-2t^2x^4$, saddle node bifurcation occurs when $\mu=0$.

\subsection{Transcritical Bifurcation}
First we consider a concrete example.
%
%
\begin{thm}\label{theorem4}
Let consider the non-autonomous differential equation \eqref{eq-18}.
\begin{equation}
\dot{x} = \mu^{2m-1}f(t) - g(t)x^{2n}, ~ m, n \in \mathbf{N} \label{eq-18}
\end{equation}
\textnormal{1)} Let assume that $f(t)$ and $g(t)$ satisfy
\begin{eqnarray}
&& \forall t \in \mathbf{R}, ~ \int_{-\infty}^t f(s)ds = +\infty, \label{eq16} \\
&& \exists T^- \in \mathbf{R}: \forall t(\leq T^-), ~ g(t)\geq r^->0, \label{eq17} \\
&& \exists \mu_0 (>0): \nonumber \\
&& \quad \forall \mu (-\mu_0<\mu \leq 0), \forall t \in \mathbf{R}, \nonumber \\
&& \qquad  \lim_{s \to -\infty} \textnormal{inf} \frac{e^{\mu^{(2m-1)}F(s)}}{\left[ (2n-1)\int_s^t g(r)e^{(2n-1)\mu^{(2m-1)}F(r)}dr \right]^{\frac{1}{2n-1}}} \geq m_{\mu}>0, \label{eq18} \\
&& \quad \forall \mu (0<\mu< \mu_0), \forall t \in \mathbf{R}, \nonumber \\
&& \qquad 0<m_{\mu} \leq x_{\mu}(t) = \frac{e^{\mu^{(2m-1)}F(t)}}{\left[ (2n-1)\int_{-\infty}^t g(r)e^{(2n-1)\mu^{(2m-1)}F(r)}dr \right]^{\frac{1}{2n-1}}} \leq M_{\mu}. \qquad  \label{eq19}
\end{eqnarray}
Here $F$ is an antiderivative of $f$. Then we have the following facts:

When $-\mu_0<\mu \leq 0$, the zero solution to \eqref{eq-18} is locally pullback attracting in $\mathbf{R}$.
When $\mu=0$, the zero solution to \eqref{eq-18} is asymptotically instable but locally pullback attracting in $\mathbf{R}^+$.
When $0<\mu< \mu_0$, the zero solution to \eqref{eq-18} is asymptotically instable and the trajectory $x_{\mu}(t)$ is locally
pullback attracting in $\mathbf{R}^+$. Furthermore
\begin{equation*}
\forall t \in \mathbf{R}, ~ x_{\mu}(t) \to 0 ~ (\mu \to 0).
\end{equation*}

\textnormal{2)} Let assume that $f(t)$ and $g(t)$ satisfy
\begin{equation}
\exists T^+: \forall t \geq T^+, g(t) \geq r^+>0, ~~ \int_t^{+\infty} f(s)ds = +\infty. \label{eq20}
\end{equation}
Then there exists a $\mu_0 (>0)$ such that the zero solution to \eqref{eq16} is forwards attracting for $-\mu_0<\mu \leq 0$
and the trajectory $x_{\mu}(t)$ is forwards attracting for $0<\mu< \mu_0$. Furthermore the additional condition
\begin{eqnarray}
&& \forall \mu<0, \forall t \in \mathbf{R}, 0<m_{\mu} \leq x_{\mu}(t) \nonumber \\
&& \qquad \qquad \qquad = \frac{e^{\mu^{(2m-1)}F(t)}}{\left[ (2n-1)\int_t^{\infty} g(r)e^{(2n-1)\mu^{(2m-1)}F(r)}dr \right]^{\frac{1}{2n-1}}} \leq M_{\mu}  \label{eq21}
\end{eqnarray}
is satisfied, then the trajectory $x_{\mu}(t)$ is asymptotically instable and pullback repelling when $-\mu_0<\mu \leq 0$.
And we have
\begin{equation*}
\forall t \in \mathbf{R}, ~ x_{\mu}(t) \to 0 ~ (\mu \to 0).
\end{equation*}
\end{thm}
\begin{proof}
When $x(s)=x_s$, the solution to \eqref{eq18} is given by
\begin{eqnarray*}
&& x(t, s, x_s) = \frac{e^{\mu^{(2m-1)}F(t)}}
{({x_s}^{-(2n-1)}e^{(2n-1)\mu^{(2m-1)}F(s)}+(2n-1)\int_t^{\infty} g(r)e^{(2n-1)\mu^{(2m-1)}F(r)}dr )^{\frac{1}{2n-1}}}.
\end{eqnarray*}
Consider the case of  $\mu<0$. For any $x_s\ne 0$ , we have $x(t, s, x_s)\rightarrow 0(s\rightarrow-\infty)$.
But we must ensure the existence of the solution  $x(t, s, x_s)$ for any $\tau\in[s,t]$.
When $x_s> 0$, if
\begin{equation}\label{eq-25}
\forall\tau\in[s,t] ,({x_s}^{-(2n-1)}e^{(2n-1)\mu^{(2m-1)}F(s)}+(2n-1)\int_t^{\infty} g(r)e^{(2n-1)\mu^{(2m-1)}F(r)}dr )^{\frac{1}{2n-1}}
\end{equation}
then the solution exists. On the other hand, from \eqref{eq20} we have
\begin{eqnarray*}
&& (2n-1)\int_s^{T^-} g(r)e^{(2n-1)\mu^{(2m-1)}F(r)}dr>0.
\end{eqnarray*}
and thus if
\begin{equation}\label{eq-26}
\forall \tau\in[T^-,t],~x_s^{-(2n-1)}e^{(2n-1)\mu^{(2m-1)}F(s)}+(2n-1)\int_{T^-}^\tau g(r)e^{(2n-1)\mu^{2m-1}F(r)}dr>0
\end{equation}
then we have \eqref{eq-25}. Here $T^-$ satisfies $\forall r\le T^-,~g(r)\ge r^->0$ and  $s\le T^-$.  
Since $\mu<0$ and $F(s)\rightarrow-\infty(s\rightarrow-\infty)$, $e^{\mu^{(2m-1)}F(s)}$ is bounded from bellow for $s\in(-\infty,T^-]$ and there exists  $\inf_{s\le T^-} e^{\mu^{(2m-1)}F(s)}$. And if 
\begin{equation}\label{eq-27}
x_s<\frac{\inf_{s\le T^-} e^{\mu^{(2m-1)}F(s)}}{\sup_{\tau\in[T^-,t]} \left|(2n-1)\int_{T^-}^\tau g(r)e^{(2n-1)\mu^{(2m-1)}F(r)}dr^\frac{1}{(2n-1)}\right|}
\end{equation}
then we have \eqref{eq-26}. The right side of \eqref{eq-27} depends only on $t$. If we denote the right side of \eqref{eq-27} by $\delta(t)$, then $\delta(t)>0$. Therefore the solution to \eqref{eq-18} with the initial value $x_s(>0)$ such that $x_s<\delta(t)$ exists for any $\tau\in [s,t]$. And when $s\to -\infty$, we have $x(t,s;x_s)\to 0$. So zero solution is locally pullback attracting in ${\mathbf R}^+$.

Consider the case of $x_s<0$. Note that $g$ is asymptotically positive when $t\to -\infty$ and $e^{(2n-1)\mu^{(2m-1)}F(r)}>0$. Thus we have
\begin{equation}\label{eq-28}
\forall t\in{\mathbf R},~\exists T_t:\forall \tau\le T_t,~\int_\tau^t e^{(2n-1)\mu^{(2m-1)}F(r)}g(r)dr>0
\end{equation}
From \eqref{eq18} and $e^{\mu^{(2m-1)}F(s)}\to-\infty(s\to-\infty)$, for this $T_t$, we have 
\begin{equation}\label{eq-29}
\exists\sigma_t;~\forall s\le\sigma_t,~\frac{e^{\mu^{(2m-1)}F(s)}}{\left((2n-1)\int_s^t e^{(2n-1)\mu^{(2m-1)}F(r)}g(r)dr\right)^\frac{1}{(2n-1)}}\ge\frac{m_\mu}{2}
\end{equation}
\begin{equation}\label{eq-30}
\frac{e^{\mu^{(2m-1)}F(s)}}{\left((2n-1)\int_s^{T_t} e^{(2n-1)\mu^{(2m-1)}F(r)}g(r)dr+(2n-1) \sup_{\tau\in[T_t,t]}\int_{T_t}^\tau e^{(2n-1)\mu^{(2m-1)}F(r)}g(r)dr\right)^{\frac{1}{(2n-1)}}}\ge\frac{m_\mu}{2}.
\end{equation}
For any $\tau\in[s,t]$, if
\begin{equation}\label{eq-31}
x_s^{-(2n-1)}e^{(2n-1)\mu^{(2m-1)}F(s)}+(2n-1)\int_s^\tau g(r)e^{(2n-1)\mu^{(2m-1)}F(r)}dr<0
\end{equation}
then the solution exists. Let $I(s,\tau)=\int_s^\tau g(r)e^{(2n-1)\mu^{(2m-1)}F(r)}dr$.

We have the following 3 cases (i), (ii) and (iii).

(i) If $I(s,\tau)<0$, then \eqref{eq13} is evident.

(ii) If $I(s,\tau)>0$ and $\tau\le T_t$, then from \eqref{eq-28} we have 
\begin{eqnarray*}
&&x_s^{-(2n-1)}e^{(2n-1)\mu^{(2m-1)}F(s)}+(2n-1)I(s,\tau)<\\
&&<x_s^{-(2n-1)}e^{(2n-1)\mu^{(2m-1)}F(s)}+(2n-1)I(s,\tau)+(2n-1)I(\tau,t)=\\
&&=x_s^{-(2n-1)}e^{(2n-1)\mu^{(2m-1)}F(s)}+(2n-1)I(s,t)
\end{eqnarray*}
 
If we prove that
\begin{equation}\label{eq-32}
|x_s|<\frac{e^{\mu^{(2m-1)}F(s)}}{\left((2n-1)\int_s^t g(r)e^{(2n-1)\mu^{(2m-1)}F(r)}dr\right)^\frac{1}{(2n-1)}}
\end{equation}
then we have $x_s^{-(2n-1)}e^{(2n-1)\mu^{(2m-1)}F(s)}+(2n-1)I(s,t)<0$. On the other hand, from \eqref{eq-29}, the right side of \eqref{eq-32} is bounded by $\frac{m_\mu}{2}$  from bellow. 

(iii) If $I(s,\tau)>0$ and $T_t<\tau\le t$, then we must prove
$$
|x_s|<\frac{e^{\mu^{(2m-1)}F(s)}}{\left((2n-1)\int_s^{T_t} e^{(2n-1)\mu^{(2m-1)}F(r)}g(r)dr+(2n-1)\int_{T_t}^\tau e^{(2n-1)\mu^{(2m-1)}F(r)}g(r)dr\right)^{\frac{1}{(2n-1)}}}.
$$
On the other hand, from \eqref{eq-30}  the right side of the above expression is bounded by $\frac{m_\mu}{2}$ from bellow. Therefore for any fixed $t$, we have the following fact: 
there exists $\sigma_t$ such that if $s\le \sigma_t$, $|x_s|<\frac{m_\mu}{2}$, then the solution exists in $[s,t]$ and zero solution is pullback attracting.

In the case of $\mu=0$, the solution of \eqref{eq-18} is given by 
$$
x(t,s;x_s)=\frac{1}{\left(x_s^{(2n-1)}+(2n-1)\int_s^t g(r)dr\right)^{\frac{1}{(2n-1)}}}.
$$

When $x_s>0$, $g$ is asymptotically positive and using the argument when $\mu<0$ we can know that zero solution is pullback attracting in ${\mathbf R}^+$.

When $x_s<0$, if $|x_s|$ is sufficiently small, then the solution to \eqref{eq-18} exists in $[s,t]$ and $x(t,s;x_s)\to 0(t\to-\infty)$. Therefore zero solution is asymptotically unstable.

In the case of $\mu>0$ we fix any $x_s\in{\mathbf R}$ and for every $t$, define $x_\mu(t)$ as follows.
$$
x_\mu(t):=\lim_{s\to-\infty} x(t,s;x_s)=\frac{e^{\mu^{(2m-1)}F(t)}}{\left[(2n-1)\int_{-\infty}^t g(r)e^{(2n-1)\mu^{(2m-1)}F(r)}dr\right]^{\frac{1}{2n-1}}}
$$

Now we prove that the orbit $x_\mu(t)$ is locally pullback attracting in ${\mathbf R}^+$. First for the orbit $x_\mu(t)$, we have $\lim_{s\to-\infty} x(t,s;x_s)=x_\mu(t)$. Now we must prove the following: 

$\forall t\in{\mathbf R},~\exists\sigma_t,~\forall s\le\sigma_t,~\forall\tau\in[s,t]$, there exists the solution $x(\tau,s,x_s)$.

In the case of $x_s<0$, if $s$ is enough large negative number, then $x(\tau,s,x_s)$ will diverge when $\tau\ge s$.

In the case of $x_s>0$, we must prove the following:
\begin{equation}\label{eq-33}
\forall \tau\in[s,t],~x_s^{-(2n-1)}e^{(2n-1)\mu^{(2m-1)}F(s)}+(2n-1)\int_s^\tau g(r)e^{(2n-1)\mu^{(2m-1)}F(r)}dr>0
\end{equation}

Assume that $x_s<x_\mu(s)$. Then we have
\begin{equation}\label{eq-34}
\exists\alpha_t>0:x_s<(1+\alpha_t)^{\frac{1}{(2n-1)}}x_\mu(s)
\end{equation}

If
\begin{equation}\label{eq-35}
\forall\tau\in[s,t],~\int_{-\infty}^\tau g(r)e^{(2n-1)\mu^{(2m-1)}F(r)}dr>\frac{\alpha_t}{1+\alpha_t}\int_{-\infty}^\tau g(r)e^{(2n-1)\mu^{(2m-1)}F(r)}dr
\end{equation}
then \eqref{eq-33} holds. Since $g$ is asymptotically positive when $t\to-\infty$, we can take $s$ such that $s\le T^-$ and $\alpha_t>0$ is properly taken, then we have \eqref{eq-35} for $\forall \tau\in[T^-,t]$.

From \eqref{eq19}, $x_\mu(t)$ satisfies $x_\mu(t)\ge m_\mu>0$ for $\forall t\in{\mathbf R}$. then there exists $x_s$ such that $0<x_s<x_\mu(s)$ and for this $x_s$, the solution exists in $[s,t]$. Thus the orbit $x_\mu(t)$ is locally pullback attracting in ${\mathbf R}^+$. 

Now we prove the asymptotical instability of zero solution.

$x(t)\equiv0$ and $x_\mu(\cdot)$ are the solutions to \eqref{eq-18} and the equation preserves order, therefore the solution such that $0<x_s<x_\mu(s)$ exists in $t\le s$. 
Thus $x_\mu(t)$ is bounded by $m_\mu$ from bellow. Let $\delta(t)=m_\mu$, then there exists the solution to \eqref{eq-18} with initial value $x_s$ such that $0<x_s<\delta(t)$. Since
$$
\lim_{t\to-\infty} \frac{e^{\mu^{(2m-1)}F(t)}}{\left(x_s^{-(2n-1)}e^{(2n-1)\mu^{(2m-1)}F(s)}+(2n-1)\int_s^t g(r) e^{(2n-1)\mu^{(2m-1)}F(r)}dr\right)^{\frac{1}{(2n-1)}}}=0,
$$
zero solution is asymptotically unstable.

We prove $x_\mu(t)\to 0(\mu\to 0)$. Fix any $t,~\varepsilon(>0)$ and take $T$ such that $\int_T^t g(r)dr>\frac{2e^{(2n-1)\mu^{(2m-1)}F(r)}}{\varepsilon}$. Then we have
\begin{eqnarray*}
&&\int_{-\infty}^t g(r) e^{(2n-1)\mu^{(2m-1)}F(r)}dr=\int_{-\infty}^T g(r) e^{(2n-1)\mu^{(2m-1)}F(r)}dr+\\
&&+\int_T^t g(r) e^{(2n-1)\mu^{(2m-1)}F(r)}dr>\int_T^t g(r) e^{(2n-1)\mu^{(2m-1)}F(r)}dr
\end{eqnarray*}

Now take $\mu$ arbitrarily small so that 
$$
\sup_{r\in[T,t]} \left|e^{(2n-1)\mu^{(2m-1)}F(r)}-1\right|<\frac{e^{(2n-1)\mu^{(2m-1)}F(r)}}{\varepsilon\int_T^t g(r)dr},
$$
then we have $\int_{-\infty}^t e^{(2n-1)\mu^{(2m-1)}F(r)}g(r)dr>\frac{e^{(2n-1)\mu^{(2m-1)}F(r)}}{\varepsilon}$ and thus $x_\mu(t)<\varepsilon$. Under the condition $\int_t^{+\infty} f(s)ds=+\infty$ and \eqref{eq20}, zero solution is locally forward attracting when $\mu<0$.

When $x_s>0$, if 
\begin{equation}\label{eq-36}
\forall \tau\in[s,t],~\left(x_s^{-(2n-1)}e^{(2n-1)\mu^{(2m-1)}F(s)}+(2n-1)\int_\tau^t g(r)e^{(2n-1)\mu^{(2m-1)}F(r)}dr\right)^{\frac{1}{(2n-1)}}>0
\end{equation}
then the solution exists. From \eqref{eq20} we have
$$
(2n-1)\int_{T^+}^t g(r) e^{(2n-1)\mu^{(2m-1)}F(r)}dr>0
$$
and therefore if
\begin{equation}\label{eq-37}
\forall \tau\in[s,T^+],~x_s^{-(2n-1)}e^{(2n-1)\mu^{(2m-1)}F(s)}+(2n-1)\int_\tau^{T^+} g(r) e^{(2n-1)\mu^{(2m-1)}F(r)}dr>0
\end{equation}
then we have \eqref{eq-36}. Here $T^+$ satisfies $\forall r\ge T^+,~g(r)\ge r^+>0$ and $s\le T^+$.  
Since $\mu<0$ and $F(s)\to+\infty(s\to+\infty)$, $e^{(2n-1)\mu^{(2m-1)}F(s)}$ is bounded from bellow in $s\in(-\infty,T^+]$ and there exists $\inf_{s\le T^+} e^{(2n-1)\mu^{(2m-1)}F(s)}$. Furthermore if 
\begin{equation}\label{eq-38}
x_s<\frac{\inf_{s\le T^+} e^{(2n-1)\mu^{(2m-1)}F(s)}}{\sup_{\tau\in[s,T^+]} \left|\left((2n-1)\int_\tau^{T^+} g(r) e^{(2n-1)\mu^{(2m-1)}F(r)}dr\right)^{\frac{1}{(2n-1)}}\right|}
\end{equation}
then we have \eqref{eq-37}. If we denote the right hand of \eqref{eq-38} by $\delta(s)$, then $\delta(s)>0$. Thus there exists the solution with initial value $x_s$ such that $0<x_s<\delta(s)$ for any $\tau\in[s,t]$ and $x(t,s;x_s)\to 0$ when $t\to+\infty$. Thus zero solution is locally forward attracting in ${\mathbf R}^+$.

Now consider the case of $x_s<0$. To do this, we need 
\begin{equation}\label{eq-39}
\forall\mu(-\mu_0<\mu\le 0),~\forall t\in{\mathbf R},~\lim_{t\to+\infty}\inf\frac{e^{(2n-1)\mu^{(2m-1)}F(s)}}{\left[(2n-1)\int_s^t g(r) e^{(2n-1)\mu^{(2m-1)}F(r)}dr\right]^{\frac{1}{2n-1}}}\ge m_\mu>0
\end{equation}
which is equivalent to \eqref{eq19}. Since $g$ is asymptotically positive when $t\to-\infty$ and $e^{(2n-1)\mu^{(2m-1)}F(s)}>0$ we have 
\begin{equation}\label{eq-40}
\forall t\in{\mathbf R},~\exists T_t:\forall \tau\ge T_t,~\int_s^\tau e^{(2n-1)\mu^{(2m-1)}F(r)}g(r)dr>0.
\end{equation}
From \eqref{eq-39} and $e^{(2n-1)\mu^{(2m-1)}F(s)}\to+\infty(s\to-\infty)$, for this $T_t$, we have 
\begin{equation}\label{eq-41}
\exists\sigma_t;~\forall t\ge\sigma_t,~\frac{e^{(2n-1)\mu^{(2m-1)}F(s)}}{\left((2n-1)\int_s^t e^{(2n-1)\mu^{(2m-1)}F(r)}g(r)dr\right)^{\frac{1}{(2n-1)}}}>\frac{m_\mu}{2}
\end{equation}
\begin{equation}\label{eq-42}
\frac{e^{(2n-1)\mu^{(2m-1)}F(s)}}{\left((2n-1)\int_{T_t}^t e^{(2n-1)\mu^{(2m-1)}F(r)}g(r)dr+(2n-1)\sup_{\tau\in[s,T_t]}\int_\tau^{T_t} e^{(2n-1)\mu^{(2m-1)}F(r)}g(r)dr\right)^{\frac{1}{(2n-1)}}}>\frac{m_\mu}{2}.
\end{equation}

If for any $\tau\in[s,t]$, 
\begin{equation}\label{eq-43}
x_s^{-(2n-1)}e^{(2n-1)\mu^{(2m-1)}F(s)}+(2n-1)\int_\tau^t g(r) e^{(2n-1)\mu^{(2m-1)}F(r)}dr<0
\end{equation}
then the solution exists. Let $I(\tau,t)=\int_\tau^t g(r) e^{(2n-1)\mu^{(2m-1)}F(r)}dr$.

Now we have the following 3 cases (i), (ii) and (iii).

(i) If $I(\tau,t)<0$, then \eqref{eq-43} is evident.

(ii) If $I(\tau,t)>0$ and $\tau\ge T_t$, then from \eqref{eq-40} we have $I(s,\tau)>0$ and thus we have
\begin{eqnarray*}
&&x_s^{-(2n-1)}e^{(2n-1)\mu^{(2m-1)}F(s)}+(2n-1)I(\tau,t)<\\
&&<x_s^{-(2n-1)}e^{(2n-1)\mu^{(2m-1)}F(s)}+(2n-1)I(s,\tau)+(2n-1)I(\tau,t)=\\
&&=x_s^{-(2n-1)}e^{(2n-1)\mu^{(2m-1)}F(s)}+(2n-1)I(s,t).
\end{eqnarray*}
If we prove that
\begin{equation}\label{eq-44}
|x_s|<\frac{e^{(2n-1)\mu^{(2m-1)}F(s)}}{\left((2n-1)\int_s^t g(r)e^{(2n-1)\mu^{(2m-1)}F(r)}dr\right)^{\frac{1}{(2n-1)}}}
\end{equation}
then we have $x_s^{-(2n-1)}e^{(2n-1)\mu^{(2m-1)}F(s)}+(2n-1)I(s,t)<0$. On the other hand, from \eqref{eq-41} $\frac{m_\mu}{2}$ is a lower bound of the right side of \eqref{eq-44}. 

(iii) If $I(\tau,t)>0$ and $\tau\le T_t$, then we must prove
\begin{equation}\label{eq-45}
|x_s|\le\frac{e^{(2n-1)\mu^{(2m-1)}F(s)}}{\left((2n-1)\int_{T_t}^t e^{(2n-1)\mu^{(2m-1)}F(r)}g(r)dr+(2n-1)\int_\tau^{T_t} e^{(2n-1)\mu^{(2m-1)}F(r)}g(r)dr\right)^{\frac{1}{(2n-1)}}}.
\end{equation}
From \eqref{eq-42} the right side of the above expression is bounded by $\frac{m_\mu}{2}$ from bellow. Therefore zero solution is locally forward attracting in ${\mathbf R}$.

In the case of $\mu=0$, zero solution is locally forward attracting. 

Now we prove that the orbit $x_\mu(t)$ is locally forward attracting when $\mu>0$.
\begin{eqnarray*}
&&\frac{1}{x^{(2n-1)}(t)}-\frac{1}{x_\mu^{(2n-1)}(t)}=e^{(2n-1)\mu^{(2m-1)}F(r)}\left[x_s^{-(2n-1)}e^{(2n-1)\mu^{(2m-1)}F(s)}+\right.
\\
&&\left.+(2n-1)\int_s^t g(r) e^{(2n-1)\mu^{(2m-1)}F(r)}dr-(2n-1)\int_{-\infty}^t g(r) e^{(2n-1)\mu^{(2m-1)}F(r)}dr\right]=
\\
&&=e^{(2n-1)\mu^{(2m-1)}[F(s)-F(t)]}\left[x_s^{-(2n-1)}-(2n-1)e^{(2n-1)\mu^{(2m-1)}F(s)}\int_{-\infty}^s g(r) e^{(2n-1)\mu^{(2m-1)}F(r)}dr\right]
\\
&&=e^{(2n-1)\mu^{(2m-1)}[F(s)-F(t)]}\left[\frac{1}{x_s^{(2n-1)}}-\frac{1}{x_\mu^{(2n-1)}(s)}\right].
\end{eqnarray*}
Therefore we have
$$
\frac{x_\mu^{(2n-1)}(t)-x^{(2n-1)}(t)}{x_\mu^{(2n-1)}(t)x^{(2n-1)}(t)}=\frac{e^{(2n-1)\mu^{(2m-1)}F(s)}}{e^{(2n-1)\mu^{(2m-1)}F(t)}}\cdot\frac{x_\mu^{(2n-1)}(s)-x_s^{(2n-1)}}{x_\mu^{(2n-1)}(s)x_s^{(2n-1)}}
$$
\begin{equation}\label{eq-46}
\left|x^{(2n-1)}(t)-x_\mu^{(2n-1)}(t)\right|=\frac{x_\mu^{(2n-1)}(t)x^{(2n-1)}(t)}{e^{(2n-1)\mu^{(2m-1)}F(t)}}\cdot\frac{e^{(2n-1)\mu^{(2m-1)}F(s)}}{x_s^{(2n-1)}x_\mu^{(2n-1)}(t)}\left|x_\mu^{(2n-1)}(t)-x_s^{(2n-1)}\right|
\end{equation}
When $x_s>0$, we prove that the solution is bounded as $t\to+\infty$.
\begin{eqnarray*}
&&x^{(2n-1)}(t)=\frac{e^{(2n-1)\mu^{(2m-1)}F(t)}}{x_s^{-(2n-1)}e^{(2n-1)\mu^{(2m-1)}F(s)}+(2n-1)\int_s^t g(r) e^{(2n-1)\mu^{(2m-1)}F(r)}dr}\le
\\
&&\le M_\mu^{(2n-1)}\frac{(2n-1)\int_{-\infty}^t g(r) e^{(2n-1)\mu^{(2m-1)}F(r)}dr}{x_s^{-(2n-1)}e^{(2n-1)\mu^{(2m-1)}F(s)}+(2n-1)\int_s^t g(r) e^{(2n-1)\mu^{(2m-1)}F(r)}dr}
\\
&&=M_\mu^{(2n-1)}\frac{(2n-1)\int_{-\infty}^s g(r) e^{(2n-1)\mu^{(2m-1)}F(r)}dr+(2n-1)\int_s^t g(r) e^{(2n-1)\mu^{(2m-1)}F(r)}dr}{x_s^{-(2n-1)}e^{(2n-1)\mu^{(2m-1)}F(s)}+(2n-1)\int_s^t g(r) e^{(2n-1)\mu^{(2m-1)}F(r)}dr}
\end{eqnarray*}

From \eqref{eq20}, for sufficiently large $t$, the second terms in numerator and the denominator are positive. Since
$$
\frac{(2n-1)\int_{-\infty}^s g(r) e^{(2n-1)\mu^{(2m-1)}F(r)}dr}{x_s^{-(2n-1)}e^{(2n-1)\mu^{(2m-1)}F(s)}}=\frac{x_s^{(2n-1)}}{x_\mu^{(2n-1)}(s)},
$$
we have
$$
\lim_{t\to+\infty}\sup x^{(2n-1)}(t)\le M_\mu^{(2n-1)}\max\left(1,\frac{x_s^{(2n-1)}}{x_\mu^{(2n-1)}(s)}\right).
$$
Therefore from \eqref{eq-46}, $x_\mu(t)$ is locally forward attracting.

In \eqref{eq21} we use transformation $\mu\to-\mu,~x\to-x,~t\to-t$, then we have \eqref{eq19}. 
Therefore using the above argument, we conclude that $x_\mu(t)$ is asymptotically unstable and locally pullback repelling. 
\end{proof}

Now we consider general equations 
\begin{equation}\label{eq-47}
\dot{x} = G(t, x, \mu).
\end{equation}
Assume that $G$ is sufficiently smooth. Then we obtain the following Taylor expansion of $G$ at $(t, 0, 0)$ as the above.
\begin{eqnarray*}
&& G(t, x, \mu) = G(t, 0, 0) + G_x(t, 0, 0)x + G_{\mu}(t, 0, 0)\mu + \frac{1}{2}G_{xx}(t, 0, 0)x^2 \\
&& \qquad + G_{x\mu}(t, 0, 0)x\mu + \frac{1}{2}G_{\mu\mu}(t, 0, 0)\mu^2 + \frac{1}{6}G_{xxx}(t, 0, 0)x^3 + \frac{1}{2}G_{xx\mu}(t, 0, 0)x^2\mu\\
&& \qquad + \frac{1}{2}G_{x\mu\mu}(t, 0, 0)x\mu^2 + \frac{1}{6}G_{\mu\mu\mu}(t, 0, 0)\mu^3 + \cdots + \frac{1}{(2m-1)!} \\
&& \qquad \times \left[ \frac{\partial^{2m-1}}{\partial x^{2m-1}}G(t, 0, 0)x^{2m-1} +C_{2m-1}^1 \frac{\partial^{2m-1}}{\partial x^{2m-2} \partial\mu}G(t, 0, 0)x^{2m-2}\mu + \cdots \right.\\
&& \qquad  \left.+ C_{2m-1}^{2m-2} \frac{\partial^{2m-1}}{\partial x \partial\mu^{2m-2}}G(t, 0, 0)x\mu^{2m-2} + \frac{\partial^{2m-1}}{\partial \mu^{2m-1}}G(t, 0, 0)\mu^{2m-1} \right] \\
&& \qquad + \frac{1}{(2m)!}\left[ \frac{\partial^{2m}}{\partial x^{2m}}G(t, 0, 0)x^{2m} +  C_{2m}^1 \frac{\partial^{2m}}{\partial x^{2m-1} \partial\mu}G(t, 0, 0)x^{2m-1}\mu + \cdots \right.\\
&& \qquad \left. + C_{2m}^{2m-1} \frac{\partial^{2m}}{\partial x \partial\mu^{2m-1}}G(t, 0, 0)x\mu^{2m-1}+\frac{\partial^{2m}}{\partial\mu^{2m}}G(t, 0, 0)\mu^{2m} \right] + \frac{1}{(2m+1)!}\\
&& \qquad \times \left[ \cdots + C_{2m+1}^{2m} \frac{\partial^{2m+1}}{\partial x \partial\mu^{2m}}G(t, 0, 0)x\mu^{2m}+\frac{\partial^{2m+1}}{\partial\mu^{2m+1}}G(t, 0, 0)\mu^{2m+1} \right].
\end{eqnarray*}
Here $m \in \mathbf{N}$.

Now assume that $G$ satisfies the following conditions:
\begin{eqnarray*}
&& \textnormal{(i)} ~ G(t, 0, \mu) = 0, ~ \forall t, \mu \in \mathbf{R}, \\
&& \textnormal{(ii)} ~ G_x(t, 0, 0) = 0, \\
&& \textnormal{(iii)} ~ \frac{\partial^2}{\partial x\partial \mu} G(t, 0, 0) = \frac{\partial^3}{\partial x\partial \mu^2} G(t, 0, 0) = \cdots = \frac{\partial^{2m-1}}{\partial x \partial \mu^{2m-2}}G(t, 0, 0) = 0. \qquad \qquad \qquad \qquad
\end{eqnarray*}
From (i) and (ii) we have $ \frac{\partial^k}{\partial \mu^k} G(t, 0, 0) = 0, ~ \forall t \in \mathbf{R}, \forall k \in \mathbf{Z}_+$
and thus $G$ is provided as follows:
\begin{eqnarray*}
&& G(t, x, \mu) = \mu^{2m-1} \left[\frac{1}{(2m)!}C_{2m}^{2m-1} \frac{\partial^{2m}}{\partial x \partial \mu^{2m-1}}G(t, 0, 0) +\frac{1}{(2m+1)!} \right. \\
&& \qquad \left. \times C_{2m+1}^{2m} \frac{\partial^{2m+1}}{\partial x \partial \mu^{2m}}G(t, 0, 0)\mu + \cdots \right]x +\left[ \frac{1}{2}G_{xx}(t, 0, 0) + \frac{1}{6}G_{xxx}(t, 0, 0)x \right. \\
&& \qquad + \frac{1}{2}G_{xx\mu}(t, 0, 0)\mu + \cdots + \frac{1}{(2m-1)!} \frac{\partial^{2m-1}}{\partial x^{2m-1}}G(t, 0, 0)x^{2m-3} + \frac{1}{(2m-1)!}\\
&& \qquad \times C_{2m-1}^1 \frac{\partial^{2m-1}}{\partial x^{2m-2} \partial\mu}G(t, 0, 0)x^{2m-4}\mu + \cdots + \frac{1}{(2m)!} \frac{\partial^{2m}}{\partial x^{2m}}G(t, 0, 0)x^{2m-2}\\
&& \qquad \left. + \frac{1}{(2m)!} C_{2m}^1 \frac{\partial^{2m}}{\partial x^{2m-1} \partial \mu}G(t, 0, 0)x^{2m-3} \mu + \cdots \right]x^2.
\end{eqnarray*}

%
%
\begin{thm}\label{theorem5}
Denote $f(t):=\frac{1}{(2m)!}C_{2m}^{2m-1} \frac{\partial^{2m}}{\partial x \partial \mu^{2m-1}}G(t, 0, 0), ~ g(t):=-\frac{1}{2}G_{xx}(t, 0, 0)$.
Then \eqref{eq-47} can be rewritten as follows:
\begin{equation}\label{eq-48}
\dot{x} = \mu^{2m-1} [f(t)+\mu\phi(t, \mu)]x - [g(t)+r(t, x, \mu)]x^2. 
\end{equation}
Here $\phi(t, 0)=\frac{1}{(2m+1)!}C_{2m+1}^{2m} \frac{\partial^{2m+1}}{\partial x \partial \mu^{2m}}G(t, 0, 0)$. Assume that
\begin{equation}
r(t, 0, 0) = 0, \label{eq23}
\end{equation}
\begin{equation}
\lim_{t \to \pm\infty}\textnormal{inf} ~ g(t)>0, \label{eq24}
\end{equation}
\begin{equation}
0<m = \lim_{t \to \pm\infty}\textnormal{inf} ~ \frac{f(t)}{g(t)} \leq  \lim_{t \to \pm\infty}\textnormal{sup} ~ \frac{f(t)}{g(t)}=M<+\infty, \label{eq25}
\end{equation}
and there exists a positive valued function $h(t)$ such that
\begin{equation}
\vert \phi(t, \mu)\vert \leq h(t), ~~ \vert r_{\mu}(t, x, \mu)\vert \leq h(t), ~~\vert r_x(t, x, \mu)\vert \leq h(t), \label{eq26}
\end{equation}
 \begin{equation}
\lim_{t \to \pm\infty}\textnormal{sup} ~ \frac{h(t)}{g(t)} \leq k. \label{eq27}
\end{equation}
Then there occurs local transcritical bifurcation when $\mu$ passes through 0.
Furthermore, when $\mu<0$, a complete orbit $x_{\mu}(t)$ is pullback repelling in $(-\varepsilon, 0)$;
when $\mu=0$, the zero solution is locally forwards attracting in $\mathbf{R}^+$ and when $\mu>0$,
pullback attracting orbit $x_{\mu}(t)$ is forwards attracting in $(0, \varepsilon)$.
\end{thm}
The idea of the proof is just the same with the theorem \ref{theorem4} and so we omit it.\\

Now assume that G satisfies the following conditions:
\begin{eqnarray*}
&& \textnormal{(iv)} ~ G(t, x, \mu) = c(\mu) \cdot G(t, x, 0) + x \cdot \frac{\partial}{\partial x}G(0, 0, \mu), \\
&& \textnormal{(v)} ~ G(t, 0, 0) = 0, \\
&& \textnormal{(vi)} ~ \frac{\partial}{\partial x}G(t, 0, 0) = \frac{\partial^2}{\partial x^2}G(t, 0, 0) = \cdots = \frac{\partial^{2n-1}}{\partial x^{2n-1}}G(t, 0, 0)=0, \\
&& \textnormal{(vii)} ~ \frac{\partial^2}{\partial x \partial\mu}G(t, 0, 0) = \frac{\partial^3}{\partial x \partial\mu^2}G(t, 0, 0) = \cdots = \frac{\partial^{2m-1}}{\partial x \partial \mu^{2m-2}}G(t, 0, 0)=0. \qquad \qquad \qquad \qquad
\end{eqnarray*}
Here $n, m \in \mathbf{N}$. Then $G$ is provided as follows:
\begin{eqnarray*}
&& G(t, x, \mu) =  \mu^{2m-1} \left[ \frac{1}{(2m)!} C_{2m}^{2m-1} \frac{\partial^{2m}}{\partial x \partial\mu^{2m-1}} G(t, 0, 0)+\frac{1}{(2m+1)!} \right. \\
&& \qquad \left. \times C_{2m+1}^{2m} \frac{\partial^{2m+1}}{\partial x \partial\mu^{2m}} G(t, 0, 0) \mu +\cdots \right]x + \left[ \frac{1}{2}x^{-(2n-2)} \mu \frac{\partial^3}{\partial x^2 \partial\mu} G(t, 0, 0) \right. \\
&& \qquad + \cdots + \frac{1}{(2m-2)!} C_{2m-2}^{1} x^{-(2n-2m+3)}\mu \frac{\partial^{2m-2}}{\partial x^{2m-3} \partial\mu} G(t, 0, 0) \\
&& \qquad + \cdots + \frac{1}{(2n)!} \frac{\partial^{2n}}{\partial x^{2n}}G(t, 0, 0) + \frac{1}{(2n)!}C_{2n}^1 x^{-1}\mu \frac{\partial^{2n}}{\partial x^{2n-1} \partial\mu}G(t, 0, 0) \\
&& \qquad \left. + \frac{1}{(2n)!}C_{2n}^2 x^{-2}\mu^2 \frac{\partial^{2n}}{\partial x^{2n-2} \partial\mu^2}G(t, 0, 0)+\cdots \right]x^{2n}, \qquad m, n \in \mathbf{N}.
\end{eqnarray*}

%
%
\begin{thm}\label{theorem6}
Denote
\begin{eqnarray*}
&& f(t):=\frac{1}{(2m)!}C_{2m}^{2m-1}\frac{\partial^{2m}}{\partial x \partial \mu^{2m-1}}G(t, 0, 0), \\
&& g(t):=-\frac{1}{(2n)!}\frac{\partial^{2n}}{\partial x \partial \mu^{2n-1}}G(t, 0, 0).
\end{eqnarray*}
Then we can rewritten \eqref{eq-47} as follows:
\begin{equation*}
\dot{x} = \mu^{2m-1}[f(t)+\mu\phi(t, \mu)] x - [g(t)+r(t, x, \mu)]x^{2n}.
\end{equation*}
Here $\phi(t, 0) = \frac{1}{(2m+1)!}C_{2m+1}^{2m}\frac{\partial^{2m+1}}{\partial x \partial \mu^{2m}}G(t, 0, 0)$. Assume that
\begin{equation*}
r(t, 0, 0) = 0, ~~ \lim_{t \to \pm\infty}\textnormal{inf} ~ g(t)>0,
\end{equation*}
\begin{equation*}
0<m = \lim_{t \to \pm\infty}\textnormal{inf} ~ \frac{f(t)}{g(t)} \leq  \lim_{t \to \pm\infty}\textnormal{sup} ~ \frac{f(t)}{g(t)}=M<+\infty,
\end{equation*}
\begin{equation*}
\vert \phi(t, \mu)\vert \leq h(t), ~~ \vert r(t, x, \mu)\vert \leq h(t) \left[\vert x\vert^{-(2n-2)} + \vert \mu \vert \right],
\end{equation*}
 \begin{equation*}
\lim_{t \to \pm\infty}\textnormal{sup} ~ \frac{h(t)}{g(t)} \leq k.
\end{equation*}
Then we have the same conclusions with the Theorem \ref{theorem5}.
\end{thm}
The proof is omitted.\\

\textbf{Example 2}. In the equation $\dot{x} = \mu^3t^2x-2t^2x^6$, transcritical bifurcation occurs when $\mu = 0$.

\subsection{Pitchfork bifurcations}

\indent

\begin{thm}\label{theorem7}
Consider the autonomous differential equation 
\begin{equation}\label{eq-59}
\dot{y}=H(t,y,\mu)
\end{equation}
Assume that the equation \eqref{eq-59} is invariant under the transformation $y\to-y$($H$ is odd function of $y$). Take the transformation $x=y^2$ to \eqref{eq-59}, then we have 
\begin{equation}\label{eq-60}
\dot{x}=G(t,x,\mu)=2yH(t,y,\mu)
\end{equation}
Assume that $G$ in \eqref{eq-60} satisfies all conditions of theorem \ref{theorem6}. Let
$$
f(t):=\frac{1}{2}\frac{1}{(2m)!}C_{2m}^{2m-1} \frac{\partial^{2m}}{\partial x \partial \mu^{2m-1}}G(t, 0, 0)
$$ 
$$
g(t):=-\frac{1}{2}\frac{1}{2}G_{xx}(t, 0, 0).
$$
All conditions of theorem \ref{theorem6} except for the limit conditions (when $t\to+\infty$) are satisfy. 
Then there occurs local pitchfork bifurcation in \eqref{eq-59} when $\mu=0$. 

On the other hand, when \eqref{eq-60} does not satisfies the conditions of theorem \ref{theorem6}, we assume that the equation \eqref{eq-60} is invariant under the transformation $x\to-x$. Take the transformation $x=z^2$ to \eqref{eq-60}, then we have 
\begin{equation}\label{eq-61}
\dot{z}=I(t,z,\mu)=2xG(t,x,\mu)
\end{equation}
Assume that \eqref{eq-61} satisfies all conditions of theorem \ref{theorem6}. Let
$$
f(t):=\frac{1}{4}\frac{1}{(2m)!}C_{2m}^{2m-1} \frac{\partial^{2m}}{\partial x \partial \mu^{2m-1}}G(t, 0, 0)
$$ 
$$
g(t):=-\frac{1}{4}\frac{1}{2}G_{xx}(t, 0, 0).
$$
All conditions of theorem \ref{theorem6} except for the limit conditions (when $t\to+\infty$) are satisfy. 
Then there occurs local pitchfork bifurcation in \eqref{eq-59} when $\mu=0$ .
\end{thm}
\begin{proof}
In fact, applying theorem \ref{theorem6}, when $-\mu_0<\mu<0$, zero solution to \eqref{eq-60} is locally pullback attracting in $(-\varepsilon,\varepsilon)$, and therefore zero solution to \eqref{eq-59} is also locally pullback attracting in $(-\varepsilon^\prime,\varepsilon^\prime)$. When $\mu=0$, zero solution to \eqref{eq-60} is asymptotically unstable and locally pullback attracting in $(0,\varepsilon)$, and so is zero solution to \eqref{eq-59}. When $0<\mu<\mu_0$, zero solution to \eqref{eq-60} is asymptotically unstable, its positive orbit $x_\mu(t)$ is locally pullback attracting in $(0,\varepsilon)$ and $x_\mu(t)\to 0(\mu\to 0)$, therefore zero solution to \eqref{eq-59} is asymptotically unstable and the orbits $y_\mu^\pm=\pm\sqrt{x_\mu(t)}$ of \eqref{eq-59} are pullback attracting in $(0,\varepsilon^\prime),~(-\varepsilon^\prime,0)$, respectively. And we have $y_\mu^\pm(t)\to 0(\mu\to 0)$. Therefore in \eqref{eq-59} there occurs local pitchfork bifurcation when $\mu=0$.

On the other hand, applying theorem \ref{theorem6} to \eqref{eq-61}, then in the case of $-\mu_0<\mu<0$  zero solution to \eqref{eq-61} is locally pullback attracting in $(-\varepsilon,\varepsilon)$ and thus zero solution to \eqref{eq-61} is locally pullback attracting in $(-\varepsilon^\prime,\varepsilon^\prime)$. For the same reason, zero solution to \eqref{eq-59} is also locally pullback attracting in $(-\varepsilon^{\prime\prime},\varepsilon^{\prime\prime})$.

In the case of $\mu=0$, zero solution to \eqref{eq-61} is asymptotically unstable and is locally pullback attracting in $(0,\varepsilon)$. Therefore so is zero solution to \eqref{eq-60}. Thus zero solution to \eqref{eq-59} is asymptotically unstable and is locally pullback attracting in $(0,\varepsilon^{\prime\prime})$.

In the case of $0<\mu<\mu_0$, zero solution to \eqref{eq-61} is asymptotically unstable, its positive orbits $z_\mu(t)$ is pullback attracting in $(0,\varepsilon)$ and $z_\mu(t)\to 0(\mu\to 0)$. Thus zero solution to \eqref{eq-60} is asymptotically unstable and the orbits $x_\mu^\pm=\pm\sqrt{z_\mu(t)}$ of \eqref{eq-60} are pullback attracting in $(0,\varepsilon^{\prime\prime}),~(\varepsilon^{\prime\prime},0)$, respectively. And we have $x_\mu^\pm(t)\to 0(\mu\to 0)$. Therefore in \eqref{eq-60} there occurs local pitchfork bifurcation when $\mu=0$. For the same reason, zero solution to \eqref{eq-59} is also asymptotically unstable and the orbits $y_\mu^\pm=\pm\sqrt{x_\mu(t)}$ of \eqref{eq-59} is locally pullback attracting in $(0,\varepsilon^{\prime\prime}),~(-\varepsilon^{\prime\prime},0)$. And we have $y_\mu^\pm(t)\to 0(\mu\to 0)$. Therefore in \eqref{eq-59} there occurs local pitchfork bifurcation when $\mu=0$. 
\end{proof}

{\bf Example 3.} Consider the autonomous differential equation
\begin{equation}\label{eq-62}
\dot{y}=\mu^3f(t)y-g(t)y^3.
\end{equation}
This equation is invariant under the transformation $y\to-y$. Take the transformation $x=y^2$. Then \eqref{eq-62} is changed to 
\begin{equation}\label{eq-63}
\dot{x}=\mu^3(2f(t))x-(2g(t))x^2.
\end{equation}
Here let $m=2,~n=1,~f(t):=2f(t),~g(t):=2g(t)$ and assume that they satisfy the conditions of theorem \ref{theorem4}. Then in \eqref{eq-63} there occurs local saddle node bifurcation when $\mu=0$. 

Then in \eqref{eq-62} there occurs local pitch fork bifurcation when $\mu=0$. In fact, applying theorem \ref{theorem4}, in the case $-\mu_0<\mu<0$, zero solution to \eqref{eq-63} is locally pullback attracting in ${\mathbf R}$ and therefore zero solution to \eqref{eq-62} is also locally pullback attracting in ${\mathbf R}$. In the case $\mu=0$ zero solution to \eqref{eq-63} is asymptotically unstable and locally pullback attracting in ${\mathbf R}^+$ and therefore so is zero solution to \eqref{eq-62}. In the case of $0<\mu<\mu_0$ zero solution to \eqref{eq-63} is asymptotically unstable, the orbit $x_\mu(t)$ is locally pullback attracting in ${\mathbf R}^+$ and $x_\mu(t)\to 0(\mu\to 0)$. Therefore zero solution to \eqref{eq-62} is asymptotically unstable and the orbits $y_\mu^\pm(t)=\pm\sqrt{x_\mu(t)}$ of \eqref{eq-62} are locally pullback attracting in ${\mathbf R}^+,~{\mathbf R}^-$, respectively. And we have $y_\mu^\pm(t)\to 0(\mu\to 0)$. Thus in \eqref{eq-62} there occurs local pitchfork bifurcation when $\mu=0$.

{\bf Example 4.} Consider the autonomous differential equation
\begin{equation}\label{eq-64}
\dot{y}=\mu^3f(t)y-g(t)y^5.
\end{equation}
This equation is invariant under the transformation $y\to-y$. Take the transformation $x=y^2$. Then \eqref{eq-64} is changed to 
$$
\dot{x}=\mu^3(2f(t))x-(2g(t))x^3.
$$
Thus \eqref{eq-64} is reduced to example 3 and therefore there occurs local pitchfork bifurcation when $\mu=0$ . (Let $f(t):=4f(t),~g(t):=4g(t)$ and apply theorem \ref{theorem4}.)

 \end{document}